\documentclass[reqno,12pt,a4paper]{amsart}

\usepackage{enumerate}
\usepackage{amstext,amsmath,amsthm,amsfonts,amssymb,amscd}
\usepackage{latexsym,mathrsfs,dsfont,euscript}
\usepackage{fullpage}
\usepackage{txfonts}
\usepackage{color}

\usepackage{hyperref} 

\newtheorem{theorem}{Theorem}

\newtheorem{lemma}[theorem]{Lemma}

\theoremstyle{definition}
\newtheorem{definition}[theorem]{Definition}

\theoremstyle{remark}

\numberwithin{equation}{section}

\DeclareMathOperator{\supp}{supp}
\newcommand{\abs}[1]{\left\vert#1\right\vert}
\newcommand{\set}[1]{\left\{#1\right\}}
\newcommand{\proin}[2]{\left<#1,#2\right>}
\newcommand{\norm}[1]{\left\Vert#1\right\Vert}

\allowdisplaybreaks
\begin{document}
\title[]{On the Calder\'{o}n-Zygmund structure of Petermichl's kernel. Weighted inequalities}
%
%


\author[]{Hugo Aimar}
\email{haimar@santafe-conicet.gov.ar}
\author[]{Ivana G\'{o}mez}
\email{ivanagomez@santafe-conicet.gov.ar}
\thanks{This work was supported by the CONICET (grant PIP-112-2011010-0877, 2012); ANPCyT-MINCyT (grants PICT-2568,2012; PICT-3631,2015); and UNL (grant CAID-50120110100371LI,2013)}
\subjclass[2010]{Primary 42B20, 42B25; Secondary 42C40}
%
%

\begin{abstract}
We show that Petermichl's dyadic operator $\mathcal{P}$ (S.~Petermichl (2000), \textit{Dyadic shifts and a logarithmic estimate for Hankel
operators with matrix symbol}) is a Calder\'{o}n-Zygmund type operator on an adequate metric normal space of homogeneous type. As a consequence of a general result on spaces of homogeneous type, we get weighted boundedness of the maximal operator $\mathcal{P}^*$ of truncations of the singular integral. We show that dyadic $A_p$ weights are the good weights for the maximal operator $\mathcal{P}^*$ of the scale truncations of $\mathcal{P}$.
\end{abstract}
\maketitle

\section{Introduction}
In \cite{Petermichl00}, Stefanie Petermichl proves a remarkable identity that provides the Hilbert kernel $\frac{1}{x-y}$ in $\mathbb{R}$ as a mean value of dilations and translations of a basic kernel defined in terms of dyadic families on $\mathbb{R}$. The basic kernel for a fixed dyadic system $\mathcal{D}$ is described in terms of Haar wavelets. Assume that $\mathcal{D}$ is the standard dyadic family on $\mathbb{R}$, i.e. $\mathcal{D}=\cup_{j\in \mathbb{Z}}\mathcal{D}^j$ with $\mathcal{D}^j=\{I^j_k: k\in \mathbb{Z}\}$ and $I^j_k=[\frac{k}{2^j},\frac{k+1}{2^j})$. Let $\mathscr{H}$ be the standard Haar system built on the dyadic intervals in $\mathcal{D}$. There is a natural bijection between $\mathscr{H}$ and $\mathcal{D}$. We shall use $\mathcal{D}$ as the index set and we shall write $h_I$ to denote the function $h_I(x)=\abs{I}^{-1/2}(\mathcal{X}_{I^-}(x)-\mathcal{X}_{I^+}(x))$ where $I^-$ and $I^+$ are the respective left and right halves of $I$, $\mathcal{X}_E$ is, as usual, the indicator function of $E$ and $\abs{E}$ denote the Lebesgue measure of the measurable set $E$. With the above notation, the basic Petermichl's operator on $L^2(\mathbb{R})$ is given by
\begin{equation}
\mathcal{P}f(x)=\sum_{I\in\mathcal{D}}\proin{f}{h_I}(h_{I^-}(x)-h_{I^+}(x)),
\end{equation}
where, as usual, $\proin{f}{h_I}=\int_{\mathbb{R}}f(y)h_I(y) dy$. Hence, at least formally, the operator $\mathcal{P}$ is defined by the nonconvolution nonsymmetric kernel
\begin{align*}
P(x,y) &= \sum_{h\in\mathcal{D}}h_I(y)(h_{I^-}(x)-h_{I^+}(x))\\
&= P^+(x,y)+P^-(x,y);
\end{align*}
with
\begin{equation}
P^+(x,y)=\sum_{I\in\mathcal{D}^+}h_I(y)(h_{I^-}(x)-h_{I^+}(x))
\end{equation}
and $\mathcal{D}^+=\{I^j_k\in\mathcal{D}: k\geq 0\}$.

Let us observe that for $x\geq 0$, $y\geq 0$ and $x\neq y$ the series
$\sum_{I\in\mathcal{D}^+}h_I(y)[h_{I^-}(x)-h_{I^+}(x)]$ is absolute convergent. In fact
\begin{align*}
\sum_{I\in\mathcal{D}^+}\abs{h_I(y)}\abs{h_{I^-}(x)-h_{I^+}(x)} &=
\sum_{I\in\mathcal{D}^+, I\supseteq I(x,y)}\frac{1}{\sqrt{\abs{I}}}\abs{h_{I^-}(x)-h_{I^+}(x)}\\
&\leq \sum_{I\in\mathcal{D}^+, I\supseteq I(x,y)}\frac{2\sqrt{2}}{\abs{I}}=\frac{4\sqrt{2}}{\abs{I(x,y)}}
\end{align*}
where $I(x,y)$ is the smallest dyadic interval in $\mathbb{R}$ containing $x$ and $y$.

The aim of this paper is twofold. First we show that $\mathcal{P}^+$ (and $\mathcal{P}^-$) the operator induced by the kernel $P^+$ (resp. $P^-$) is of Calder\'{o}n--Zygmund type in the normal space of homogeneous type $\mathbb{R}^+$ (resp. $\mathbb{R}^-$) with the dyadic ultrametric $\delta(x,y)=\inf\{\abs{I}: x,y\in I \textrm{ and } I\in\mathcal{D}\}$ and Lebesgue measure. Second, by an application of the known weighted norm inequalities for singular integrals in normal spaces of homogeneous type, we show that the operator $\mathcal{P}^*f(x)=\sup_{\{l,m\in \mathbb{Z}\}}\abs{\sum_{\{I\in\mathcal{D}^+,2^l\leq\abs{I}<2^m\}}\proin{f}{h_I}(h_{I^-}(x)-h_{I^+}(x))}$ is bounded on $L^p(\mathbb{R}^+,wdx)$ if and only if $w\in A_p^{dy}(\mathbb{R}^+)$ when $1<p<\infty$.

In \S\ref{sec:PetermichlOperatorAsCZ} we prove that $\mathcal{P}^+$ is of Calder\'{o}n--Zygmund in an adequate space of homogeneous type. In Section~\ref{sec:weights} we give the characterization of the dyadic weights as those for which the maximal operator of the scale truncations of $\mathcal{P}^+$ is bounded in $L^p(\mathbb{R}^+,wdx)$ for $1<p<\infty$.

\section{Petermichl's operator as a Calder\'{o}n--Zygmund operator}\label{sec:PetermichlOperatorAsCZ}
Following \cite{MeyCoif97}, a linear and continuous operator $T:\mathscr{D}(\mathbb{R}^n)\to\mathscr{D}'(\mathbb{R}^n)$, with $\mathscr{D}$ and $\mathscr{D}'$ the test functions and the distributions on $\mathbb{R}^n$, is a Calder\'{o}n-Zygmund operator if there exists $K\in L^1_{loc}(\mathbb{R}^n\times \mathbb{R}^n\setminus\Delta)$ where $\Delta$ is the diagonal of $\mathbb{R}^n\times \mathbb{R}^n$ such that
\begin{enumerate}
\item there exists $C_0>0$ with
\begin{equation*}
\abs{K(x,y)}\leq \frac{C_0}{\abs{x-y}^n},\quad x\neq y;
\end{equation*}
\item\label{item:KregularityReal} there exist $C_1$ and $\gamma>0$ such that
\begin{enumerate}[(\ref{item:KregularityReal}.a)]
\item $\abs{K(x',y)-K(x,y)}\leq C_1\dfrac{\abs{x'-x}^{\gamma}}{\abs{x-y}^{n+\gamma}}$ when $2\abs{x'-x}\leq\abs{x-y}$;
\item $\abs{K(x,y')-K(x,y)}\leq C_1\dfrac{\abs{y'-y}^{\gamma}}{\abs{x-y}^{n+\gamma}}$ when $2\abs{y'-y}\leq\abs{x-y}$;
\end{enumerate}
\item $T$ extends to $L^2(\mathbb{R}^n)$ as a continuous linear operator;
\item for $\varphi$ and $\psi\in\mathscr{D}(\mathbb{R}^n)$ with $\supp\varphi\cap\supp\psi=\emptyset$ we have
\begin{equation*}
\proin{T\varphi}{\psi}=\iint_{\mathbb{R}^n\times \mathbb{R}^n}K(x,y)\varphi(x)\psi(y) dx dy.
\end{equation*}
\end{enumerate}
With a little effort the notions of Calder\'{o}n-Zygmund operator and Calder\'{o}n-Zygmund kernel $K$ (i.e. satisfying (1) and (2)) can be extended to normal metric spaces of homogeneous type. Even when the formulation can be stated in quasi-metric spaces for our application it shall be enough the following context. Let $(X,d)$ be a metric space. If there exists a Borel measure $\mu$ on $X$ such that for some constants $0<\alpha\leq\beta<\infty$ such that the inequalities $\alpha r\leq \mu(B(x,r))\leq \beta r$ hold for every $r>0$ and every $x\in X$, we shall say that $(X,d,\mu)$ is a normal space. As usual $B(x,r)=\{y\in X: d(x,y)<r\}$. In particular, $(X,d,\mu)$ is a space of homogeneous type in the sense of \cite{CoWe71}, \cite{MaSe79Lip}, \cite{MaSe79Decomp}, \cite{Aimar85}, and many problems of harmonic analysis find there a natural place to be solved.

In this setting in \cite{MaSe79Lip} a fractional order inductive limit topology is given to the space of compactly supported Lipschitz $\gamma$ functions ($0<\gamma<1$). We shall still write $\mathscr{D}=\mathscr{D}(X,d)$ to denote this test functions space. And $\mathscr{D}'=\mathscr{D}'(X,d)$ its dual, the space of distributions. So, the extension of the definition of Calder\'{o}n-Zygmund operators to this setting becomes natural.

\begin{definition}\label{def:CalderonZygmundOperator}
Let $(X,d,\mu)$ be a normal metric measure space such that continuous functions are dense in $L^1(X,\mu)$. We say that a linear and continuous operator $T:\mathscr{D}\to\mathscr{D}'$ is Calder\'{o}n-Zygmund on $(X,d,\mu)$ if there exists $K\in L^1_{loc}(X\times X\setminus\Delta)$, where $\Delta$ is the diagonal in $X\times X$, such that
\begin{enumerate}[(i)]
\item there exists $C_0>0$ with
\begin{equation*}
\abs{K(x,y)}\leq\frac{C_0}{d(x,y)},\quad x\neq y;
\end{equation*}
\item\label{item:KregularityMetric} there exist $C_1>0$ and $\gamma>0$ such that
\begin{enumerate}[(\ref{item:KregularityMetric}.a)]
\item $\abs{K(x',y)-K(x,y)}\leq C_1\dfrac{d(x',x)^{\gamma}}{d(x,y)^{1+\gamma}}$ when $2d(x',x)\leq d(x,y)$;
\item $\abs{K(x,y')-K(x,y)}\leq C_1\dfrac{d(y,y')^{\gamma}}{d(x,y)^{1+\gamma}}$ when $2d(y',y)\leq d(x,y)$;
\end{enumerate}
\item $T$ extends to $L^2(X,\mu)$ as a continuous linear operator;
\item for $\varphi$ and $\psi\in\mathscr{D}$ with $d(\supp\varphi,\supp\psi)>0$ we have
\begin{equation*}
\proin{T\varphi}{\psi}=\iint_{X\times X}K(x,y)\varphi(x)\psi(y) d(\mu\times\mu)(x,y).
\end{equation*}
\end{enumerate}
\end{definition}

Our first result shows that $\mathcal{P}^+$ and $\mathcal{P}^-$ are Calder\'{o}n-Zygmund operators. In what follows we shall keep using $P$ for $P^+$ and $\mathcal{P}$ for $\mathcal{P}^+$.
\begin{theorem}\label{thm:PetermichlOperatorisCZ}
There exists a metric $\delta$ on $\mathbb{R}^+=\{x: x\geq 0\}$ such that $(\mathbb{R}^+,\delta,\abs{\cdot})$ is a normal space where $\delta$-continuous functions are dense in $L^1(\mathbb{R}^+,dx)$ and $P$ can be written, for $x\neq y$ both in $\mathbb{R}^+$, as
\begin{equation}
P(x,y)=\dfrac{\Omega(x,y)}{\delta(x,y)},
\end{equation}
where $\Omega$ is bounded and $\delta$-smooth. Moreover, $\mathcal{P}$ is a Calder\'{o}n-Zygmund operator on $(\mathbb{R}^+,\delta,\abs{\cdot})$.
\end{theorem}
\begin{proof}
For $x\neq y$ two points in $\mathbb{R}^+$, define $\delta(x,y)=\inf\{\abs{I}: x,y\in I\in\mathcal{D}\}$. Define also $\delta(x,x)=0$ for every $x\in \mathbb{R}^+$. It is easy to see that $\delta$ is an ultra-metric on $\mathbb{R}^+$. This means that the triangle inequality improves to $\delta(x,z)\leq\sup\{\delta(x,y),\delta(y,z)\}$ for every $x$, $y$ and $z\in \mathbb{R}^+$. Notice that $\abs{x-y}\leq\delta(x,y)$ but they are certainly not equivalent. Also, for $x\in \mathbb{R}^+$ and $r>0$ given, taking $m\in \mathbb{Z}$ such that $2^{-m}<r\leq 2^{-m+1}$ we see that $B_{\delta}(x,r)=\{y\in \mathbb{R}^+: \delta(x,y)<r\}=\{y\in \mathbb{R}^+:\delta(x,y)\leq 2^{-m}\}=I^m_{k(x)}$, where $k(x)$ is the only index $k\in \mathbb{N}\cup\{0\}$ such that $x\in I^m_k$. Hence the Lebesgue measure of $B_{\delta}(x,r)$ is that of the interval $I^m_{k(x)}$. Precisely, $\abs{B_{\delta}(x,r)}=2^{-m}$. So that $\frac{r}{2}\leq \abs{B_{\delta}(x,r)}<r$, for every $x\in \mathbb{R}^+$ and every $r>0$. In terms of our above definitions $(\mathbb{R}^+,\delta,\abs{ \cdot })$ is a normal metric space. The integrability properties of powers of $\delta$ resemble completely those, of the powers of $x$. In fact, for fixed $x\in \mathbb{R}^+$, the function of $y\in \mathbb{R}^+$ given by $1/\delta^{\alpha}(x,y)$ is integrable inside a $\delta$-ball when $\alpha<1$. It is integrable outside a $\delta$-ball when $\alpha>1$. In particular, $1/\delta(x,y)$ is neither locally nor globally integrable on $\mathbb{R}^+$.

Notice now that real valued simple functions built on the dyadic intervals are continuous as functions defined on $(\mathbb{\mathbb{R}}^+,\delta)$. In fact, for $I\in\mathcal{D}$ we have that $\abs{\mathcal{X}_I(x)-\mathcal{X}_I(y)}$ equals zero for $x$ and $y$ in $I$ or for $x$ and $y$ outside $I$. Assume that $x\in I$ and $y\notin I$, then  $\delta(x,y)\geq 2\abs{I}$. So that $\abs{\mathcal{X}_I(x)-\mathcal{X}_I(y)}\leq\delta(x,y)(2\abs{I})^{-1}$ for every $x$ and $y\in \mathbb{R}^+$. In other words, for $I\in\mathcal{D}$, $\mathcal{X}_I$ is Lipschitz with respect to $\delta$ with constant  $(2\abs{I})^{-1}$. Hence $\delta$-continuous functions are dense in $L^1(\mathbb{R}^+,dx)$.

The operator $\mathcal{P}$ is actually defined as an operator in $L^2(\mathbb{R}^+,dx)$. For $f\in L^2(\mathbb{R}^+,dx)$,
\begin{align*}
\mathcal{P}f(x) &= \sum_{I\in\mathcal{D}^+}\proin{f}{h_I}(h_{I^-}(x)-h_{I^+}(x))\\
&= \sum_{I\in\mathcal{D}^+}\proin{f}{h_I}h_{I^-}(x)-\sum_{I\in\mathcal{D}^+}\proin{f}{h_I}h_{I^+}(x).
\end{align*}
Hence $\norm{\mathcal{P}f}_2^2\leq 2\sum_{I\in\mathcal{D}^+}\abs{\proin{f}{h_I}}^2=2\norm{f}_2^2$, which proves (iii) in Definition~\ref{def:CalderonZygmundOperator}. In particular, if $\varphi$ is a simple function built on the dyadic intervals, we see that $\mathcal{P}\varphi\in L^2(\mathbb{R}^+,dx)$. So that when $\psi$ is another simple function such that $\delta(\supp\varphi,\supp\psi)>0$, the two variables function $F(x,y)=\varphi(x)\psi(y)$ is simple in $\mathbb{R}^+\times \mathbb{R}^+$ and for some $\varepsilon>0$, $\supp F\cap\{\delta<\varepsilon\}=\emptyset$, we have that, since only a finite subset of $\mathcal{D}^+$ is actually involved,
\begin{align*}
\iint_{\mathbb{R}^+\times \mathbb{\mathbb{R}}^+}&\left(\sum_{I\in\mathcal{D}^+}h_I(y)[h_{I^-}(x)-h_{I^+}(x)]\right)\varphi(y)\psi(x)dy dx\\
&= \int_{x\in \mathbb{R}^+}\left(\int_{y\in \mathbb{R}^+}P(x,y)\varphi(y)\right)\psi(x) dx\\
&=\int_{x\in \mathbb{R}^+}\mathcal{P}\varphi(x)\psi(x) dx\\
&= \proin{\mathcal{P}\varphi}{\psi}.
\end{align*}
Hence $P(x,y)=\sum_{I\in\mathcal{D}^+}h_{I}(y)[h_{I^-}(x)-h_{I^+}(x)]$ is the kernel for $\mathcal{P}$. Let us now show that $P(x,y)=\frac{\Omega(x,y)}{\delta(x,y)}$ for $x\neq y$.
 For $J\in\mathcal{D}^+$ define
\begin{equation*}
\Omega_J(x,y)=\Theta_J^1(y)\Theta^2_J(x)
\end{equation*}
where
\begin{eqnarray*}
  \Theta^1_J(y)&=&\mathcal{X}_{J^-}(y)-\mathcal{X}_{J^+}(y)   \\
  \Theta^2_J(x)&=&(\mathcal{X}_{J^{-+}}(x)+\mathcal{X}_{J^{+-}}(x))-(\mathcal{X}_{J^{--}}(x)+\mathcal{X}_{J^{++}}(x)).
\end{eqnarray*}
Let us denote with $I(x,y)$ the smallest interval containing $x$ and $y$, then we have
\begin{equation*}
P(x,y)=\sum_{I\in\mathcal{D}^+}h_I(y)[h_{I^-}(x)-h_{I^+}(x)]=\sqrt{2}\sum_{I\in\mathcal{D}^+, I\supseteq I(x,y)}
\frac{1}{\abs{I}}\Omega_I(x,y).
\end{equation*}
Since $\abs{I(x,y)}=\delta(x,y)$ and in the last series we are adding on all the dyadic ancestors of $I(x,y)$, including $I(x,y)$
itself,
\begin{equation*}
P(x,y)=\frac{\sqrt{2}}{\delta(x,y)}\sum_{m=0}^{\infty}\frac{1}{2^m}\Omega_{I^{(m)}(x,y)}(x,y)=\frac{\Omega(x,y)}{\delta(x,y)}
\end{equation*}
with $I^{(m)}(x,y)$ the $m$-th ancestor of $I(x,y)$ and
\begin{equation*}
\Omega(x,y)=\sqrt{2}\sum_{m=0}^{\infty}2^{-m}\Omega_{I^{(m)}(x,y)}(x,y).
\end{equation*}
Hence (i) in Definition~\ref{def:CalderonZygmundOperator} holds with $C_0=2^{5/2}$.

Let us check (ii.a). Let $x$, $y$ and $x'\in \mathbb{R}^+$ be such that $\delta(x,x')\leq\frac{1}{2}\delta(x,y)$. Let $I(x,y)$ be the smallest dyadic
interval containing $x$ and $y$. Then $\abs{I(x,y)}=\delta(x,y)$. In a similar way $\abs{I(x,x')}=\delta(x,x')$ and
$\abs{I(x',y)}=\delta(x',y)$. Since those three intervals are all dyadic and since $\abs{I(x,x')}\leq\frac{1}{2}\abs{I(x,y)}$,
 we necessarily must have that $x'$ belongs to the same half of $I(x,y)$ as $x$ does. Hence $I(x',y)=I(x,y)$ and certainly
 also are the same all the ancestors $I^{(m)}(x',y)=I^{(m)}(x,y)$. Now,
\begin{align*}
\frac{1}{\sqrt{2}}\abs{P(x',y)-P(x,y)}&=\abs{\frac{\Omega(x',y)}{\delta(x',y)}-\frac{\Omega(x,y)}{\delta(x,y)}}\\
&\leq \frac{\abs{\Omega(x',y)-\Omega(x,y)}}{\delta(x,y)}
+\abs{\Omega(x',y)}\abs{\frac{1}{\delta(x',y)}-\frac{1}{\delta(x,y)}}\\
&=I+II.
\end{align*}
In order to estimate $I$, let us first explore the $\delta$-regularity of each $\Omega_J$. Let us prove that
\begin{enumerate}[(a)]
\item for fixed $y\in \mathbb{R}^+$ we have that $\abs{\Omega_J(x',y)-\Omega_J(x,y)}\leq \frac{8}{\abs{J}}\delta(x,x')$; and
\item for fixed $x\in \mathbb{R}^+$, $\abs{\Omega_J(x,y')-\Omega_J(x,y)}\leq \frac{2}{\abs{J}}\delta(y,y')$.
\end{enumerate}
Let us check (a). The regularity in the second variable is similar. Since the indicator function of a dyadic interval $I$ is $\delta$-Lipschitz with constant $\frac{1}{2\abs{I}}$, we have
\begin{align*}
\abs{\Omega_J(x',y)-\Omega_J(x,y)}&=\abs{\Theta^1_J(y)(\Theta^2_J(x')-\Theta^2_J(x))}\\
&=\abs{\Theta^2_J(x')-\Theta^2_J(x)}\\
&\leq \abs{\mathcal{X}_{J^{-+}}(x')-\mathcal{X}_{J^{-+}}(x)}+\abs{\mathcal{X}_{J^{+-}}(x')-\mathcal{X}_{J^{+-}}(x)}+\\
&\phantom{\leq\mathcal{X}_{J^{-+}}} 
+\abs{\mathcal{X}_{J^{--}}(x')-\mathcal{X}_{J^{--}}(x)} +\abs{\mathcal{X}_{J^{++}}(x')-\mathcal{X}_{J^{++}}(x)}\\
&\leq 4\frac{4}{2\abs{J}}\delta(x,x').
\end{align*}
Since the series defining $\Omega$ is absolutely convergent, from the above remarks, we have
\begin{align*}
I &\leq \frac{1}{\delta(x,y)}\sum_{m=0}^{\infty}2^{-m}\abs{\Omega_{I^{(m)}(x',y)}(x',y)-\Omega_{I^{(m)}(x,y)}(x,y)}\\
&=\frac{1}{\delta(x,y)}\sum_{m=0}^{\infty}2^{-m}\abs{\Omega_{I^{(m)}(x,y)}(x',y)-\Omega_{I^{(m)}(x,y)}(x,y)}\\
&\leq\frac{8}{\delta(x,y)}\sum_{m=0}^{\infty}2^{-m}\frac{\delta(x,x')}{\abs{I^{(m)}(x,y)}}\\
&=16\frac{\delta(x,x')}{\delta^2(x,y)}.
\end{align*}
Let us estimate \textit{II}. Since $\abs{\Omega}$ is bounded above by $2$ and $\delta$ is a metric on $\mathbb{R}^+$, we have
\begin{equation*}
II\leq 2\frac{\abs{\delta(x,y)-\delta(x',y)}}{\delta(x,y)\delta(x',y)}\leq 2\frac{\delta(x,x')}{\delta(x,y)\delta(x',y)}
\end{equation*}
as we already observed, under the current conditions, $\delta(x',y)=\delta(x,y)$. And we get the desired type estimate
$II\leq 2\frac{\delta(x,x')}{\delta^(x,y)}$. Hence $\abs{P(x',y)-P(x,y)}\leq\sqrt{2}\frac{14}{3}\frac{\delta(x,x')}{\delta^2(x,y)}$
when $\delta(x,x')\leq \tfrac{1}{2}\delta(x,y)$.

The analogous procedure, using (b) and a similar geometric consideration for $x$, $y$, $y'$ with
$\delta(y,y')\leq\frac{1}{2}\delta(x,y)$ gives
$$\abs{P(x,y')-P(x,y)}\leq\sqrt{2}12\frac{\delta(y,y')}{\delta^2(x,y)}.$$
\end{proof}

The next result contains some additional properties of $P$ that shall be used in the next section  in order to get weighted inequalities for the maximal operator of the truncations of $\mathcal{P}$.

As usual, for Calder\'{o}n-Zygmund operators, the truncations of the kernel and the associated maximal operator play a central role in the analysis of the boundedness properties of the operator. For $0<\varepsilon<R<\infty$ set
\begin{equation*}
P_{\varepsilon,R}(x,y)=\mathcal{X}_{\{\varepsilon\leq\delta(x,y)<R\}}P(x,y)=\mathcal{X}_{\{\varepsilon\leq\delta(x,y)<R\}}
\frac{\Omega(x,y)}{\delta(x,y)}.
\end{equation*}
Sometimes, for example when $P$ acts on $L^p(\mathbb{R}^+,dx)$ with $p>1$, only the local truncation about the diagonal is actually needed. For $\varepsilon>0$, $P_{\varepsilon,\infty}(x,y)=\mathcal{X}_{\{\delta(x,y)\geq\varepsilon\}}(x,y)P(x,y)$. Since the original form of Petermichl's kernel is provided in terms of the Haar--Fourier analysis, a scale truncation is still possible and natural. For $l<m$ both in $\mathbb{Z}$ we consider also the scale truncation of $P$ between $2^l$ and $2^m$. In other words,
\begin{equation*}
P^{l,m}(x,y)=\sum_{\{I\in\mathcal{D}^+:2^l\leq\abs{I}<2^m\}}h_I(y)[h_{I^-}(x)-h_{I^+}(x)].
\end{equation*}
Since $\delta$ takes only dyadic values, $P_{\varepsilon,R}$ can also be written as $P_{2^\lambda,2^\mu}$ for $\lambda$ and $\mu\in \mathbb{Z}$. For simplicity we shall write $P_{\lambda,\mu}$ to denote $P_{2^\lambda,2^\mu}$. Hence in our notation the distinction between the two truncations is only positional: $P^{l,m}$ is scale truncation; $P_{l,m}$ is metric truncation. Let us compare these two kernels and the operators induced by them. The calligraphic versions  $\mathcal{P}^{l,m}$ and $\mathcal{P}_{l,m}$ denote the operators induced by $P^{l,m}$ and $P_{l,m}$ respectively.

In the next statement we use two notations for the ancestrality of a dyadic interval. Given $I\in\mathcal{D}^+$, $I^{(n)}$ denotes, as before, the $n$-th ancestor of $I$. Instead $\widehat{I}^j$ denotes the only, if any, ancestor of $I$ in the level $\mathcal{D}^j$ of the dyadic interval. For instance if $I=[\tfrac{3}{2},2)$, then $I^{(1)}=[1,2)$, $I^{(2)}=[0,2)$, $\widehat{I}^0=[1,2)$, $\widehat{I}^3=[0,8)$.

\begin{lemma}\label{lem:PlmQlm}
Let $l$ and $m$ in $\mathbb{Z}$ with $l<m$. Then
\begin{enumerate}
\item $P^{l,m}(x,y)=P_{l,m}(x,y)+Q_{l,m}(x,y)$, where
\begin{equation*}
Q_{l,m}(x,y)=
\left\{
  \begin{array}{ll}
    0, & \hbox{for $\delta(x,y)\geq 2^m$;} \\
    \sqrt{2}\sum\limits_{j=l}^{m-1}2^{-j}\Omega_{\widehat{I}^j(x,y)}(x,y), & \hbox{for $0<\delta(x,y)<2^l$;} \\
    -\frac{\sqrt{2}}{\delta(x,y)}\sum\limits_{n=\log_2\tfrac{2}{\delta(x,y)}}^\infty 2^{-n}\Omega_{I^{(n)}(x,y)}(x,y), & \hbox{when $2^l\leq\delta(x,y)<2^m$.}
  \end{array}
\right.
\end{equation*}
\item $P^{l,m}$ belongs to $L^1(\mathbb{R}^+,dx)$ in each variable when the other variable remains fixed. Moreover
\begin{equation*}
\int_{y\in\mathbb{R}^+}P^{l,m}(x,y)dx=\int_{y\in\mathbb{R}^+}P^{l,m}(x,y)dy=0.
\end{equation*}
\item $\abs{Q_{l,m}(x,y)}\leq 2\sqrt{2}\left(2^{-l}\mathcal{X}_{\{\delta(x,y)<2^l\}}(x,y)+2^{-m}\mathcal{X}_{\{\delta(x,y)<2^m\}}\right)$.

\item The inequality $\abs{\int_{y\in \mathbb{R}^+}Q_{l,m}(x,y)dy}\leq 2\sqrt{2}$ holds for every $l, m$ in $\mathbb{Z}$
and every $x\in \mathbb{R}^+$.

\item The sequence $\int_{y\in \mathbb{R}^+}Q_{l,0}(x,y)dy$ converges uniformly in $x\in \mathbb{R}^+$ for $l$ tends to $-\infty$.
\end{enumerate}
\end{lemma}

\begin{proof}
Let us rewrite together the two truncations of $P$ for the same values of $l$ and $m$ with $l<m$,
\begin{eqnarray*}
   P^{l,m}(x,y)&=& \sum_{I\in\mathcal{D}^+, 2^l\leq\abs{I}<2^m}h_I(y)[h_{I^{-}}(x)-h_{I^{+}}(x)]; \\
   P_{l,m}(x,y)&=& \mathcal{X}_{\{2^l\leq\delta(x,y)<2^m\}}(x,y)\frac{\Omega(x,y)}{\delta(x,y)}
\end{eqnarray*}
with $\Omega(x,y)=\sqrt{2}\sum_{n=0}^{\infty}2^{-n}\Omega_{I^{(n)}(x,y)} (x,y)$. Let us compute
$P^{l,m}(x,y)$ for the  three bands around the diagonal $\Delta$ of $\mathbb{R}^+\times \mathbb{R}^+$ determined by $2^l$
and $2^m$. First, assume that $0<\delta(x,y)<2^l$. Then
\begin{equation*}
P^{l,m}(x,y)=\sqrt{2}\sum_{\substack{I\in\mathcal{D}^+\\ 2^l\leq\abs{I}<2^m}}\frac{1}{\abs{I}}\Omega_I(x,y).
\end{equation*}
Since $\supp\Omega_I\subset I\times I$, once $(x,y)$ is given, with $\delta(x,y)<2^l$, the sum above is performed
only on those dyadic intervals $I$ for which $2^l\leq\abs{I}<2^m$ that contain $I(x,y)$; the smallest dyadic interval containing
both $x$ and $y$. Hence
\begin{equation*}
P^{l,m}(x,y) =\sqrt{2}\sum_{j=l}^{m-1}\frac{1}{2^j}\Omega_{\widehat{I}^j(x,y)}(x,y)
=Q_{l,m}(x,y)\\
=Q_{l,m}(x,y)+P_{l,m}(x,y)
\end{equation*}
in the $\delta$-strip $\{(x,y):\mathbb{R}^+\times \mathbb{R}^+:\delta(x,y)<2^l\}$. Second, assume that $\delta(x,y)\geq 2^m$. Then
no dyadic interval $I$ containing both $x$ and $y$ has a measure less than $2^m$. So that $P^{l,m}$ vanishes when $\delta(x,y)\geq 2^m$
and again $P^{l,m}=Q_{l,m}+P_{l,m}$. The third and last case to be considered is when $2^l\leq\delta(x,y)<2^m$. Again the non-vanishing
condition for $\Omega_I(x,y)$ requires $I\supseteq I(x,y)$, hence
\begin{equation*}
P^{l,m}(x,y)=\sqrt{2}\sum_{\substack{I\in\mathcal{D}\\ \abs{I}<2^m\\ I\supseteq I(x,y)}}\frac{1}{\abs{I}}\Omega_I(x,y).
\end{equation*}
Since $I\supseteq I(x,y)$ then, in the above sum, $I$ has to be an ancestor of $I(x,y)$. Hence $\abs{I}=2^n\abs{I(x,y)}=2^n\delta(x,y)$
for some $n=0,1,2,\ldots$ The upper restriction on the measure of $I$, $\abs{I}<2^m$, provides an upper bound for $n$. In fact, since
$2^m>\abs{I}=2^n\delta(x,y)$, $n\leq(\log_2 2^m\delta^{-1}(x,y))-1$. Notice that $2^m\delta^{-1}(x,y)$ is an integral power of $2$, so that
$\log_2 2^m\delta^{-1}(x,y)\in \mathbb{Z}$. Hence
\begin{align*}
P^{l,m}&=\frac{\sqrt{2}}{\delta(x,y)}\sum_{n=0}^{\log_2\tfrac{2^m}{\delta(x,y)}-1}\frac{1}{2^n}\Omega_{I^{(n)}(x,y)}(x,y)\\
&=\frac{\sqrt{2}}{\delta(x,y)}\Biggl(\Omega(x,y)-\sum_{n=\log_2\tfrac{2^m}{\delta(x,y)}}^{\infty}\frac{1}{2^n}\Omega_{I^{(n)}(x,y)}(x,y)\Biggr)\\
&=P_{l,m}(x,y)+Q_{l,m}(x,y),
\end{align*}
and (1) is proved.

In order to prove (2), notice that for $x$ fixed $P^{l,m}(x,\cdot)$ is a finite linear combination of Haar functions in the variable $y$.
Hence $P^{l,m}(x,\cdot)$ is an $L^1(\mathbb{R}^+,dx)$ function and its integral in $y$ vanishes, since each Haar function has mean value zero.
An analogous argument hold for $y$ fixed and $P^{l,m}(\cdot,y)$.

Let us get the bound in (3). We only have to check it in the bands $\{\delta(x,y)<2^l\}$ and $\{2^l\leq\delta(x,y)<2^m\}$. Let us first
take $\delta(x,y)<2^l$. Then
\begin{equation*}
\abs{Q_{l,m}(x,y)}=\sqrt{2}\abs{\sum_{j=l}^{m-1}2^{-j}\Omega_{\widehat{I}^j(x,y)}(x,y)}\leq \sqrt{2}\sum_{j=l}^m2^{-j}\leq 2\sqrt{2}2^{-l},
\end{equation*}
as desired. Assume now that $2^l\leq\delta(x,y) <2^m$. Then
\begin{equation*}
\abs{Q_{l,m}(x,y)}\leq\sqrt{2}\frac{1}{\delta(x,y)}\sum_{n=\log_2\tfrac{2^m}{\delta(x,y)}}^{\infty}2^{-n}=2\sqrt{2}\frac{1}{\delta(x,y)}
\frac{\delta(x,y)}{2^m}=2\sqrt{2}2^{-m}.
\end{equation*}

For the proof of (4) notice that from (3) we have that, for fixed $x$ and fixed $l$ and $m$, as a function of $y$, $Q_{l,m}(x,y)$, and hence
 $P_{l,m}(x,y)$, is integrable. Then
\begin{equation*}
\Biggl|\int\limits_{y\in \mathbb{R}^+}Q_{l,m}(x,y)dy\Biggr|
\leq 2\sqrt{2}\int\limits_{y\in \mathbb{R}^+}\left\{2^{-l}\mathcal{X}_{\set{\delta(x,y)<2^l}}(x,y)
+2^{-m}\mathcal{X}_{\set{\delta(x,y)<2^m}}(x,y)\right\}dy
=2\sqrt{2}.
\end{equation*}

Let us prove (5). From the expression in (1) for $Q_{l,0}$, we have
\begin{align*}
&\int\limits_{y\in \mathbb{R}^+}Q_{l,0}(x,y)dy =\sqrt{2}\int\limits_{B_\delta(x,2^l)}\Biggl(\sum_{j=l}^{-1}2^{-j}
\Omega_{\widehat{I}^j(x,y)}(x,y)\Biggr)dy+\\
&\phantom{\int\limits_{y\in \mathbb{R}^+}Q_{l,0}(x,y)dy =}-\sqrt{2}\int\limits_{B_\delta(x,1)\setminus B_\delta(x,2^l)}\frac{1}{\delta(x,y)}
\Biggl(\sum_{n=\log_2\tfrac{1}{\delta(x,y)}}^{\infty}\frac{1}{2^n}\Omega_{I^{(n)}(x,y)}(x,y)\Biggr)dy\\
&=\sqrt{2}\sum_{j=l}^{-1}2^{-j}
\int\limits_{B_\delta(x,2^l)}\Omega_{\widehat{I}^j(x,y)}(x,y)dy
-\sqrt{2}\sum_{i=l}^{-1}2^{-i}\int\limits_{\{y:\delta(x,y)=2^i\}}
\Biggl(\sum_{n=-i}^{\infty}\frac{1}{2^n}\Omega_{I^{(n)}(x,y)}(x,y)\Biggr)dy\\
&=\sqrt{2}\, \Biggl(\sum_{j=l}^{-1}2^{-j}2^l\widehat{\sigma}_{l,j}(x)-\tfrac{1}{2}\sum_{i=l}^{-1}
2^{-i}\sum_{n=-i}^{\infty}2^{-n}2^i\sigma_{n,i}(x)\Biggr),
\end{align*}
where $\widehat{\sigma}_{l,j}(x)=\fint_{B_\delta(x,2^l)}\Omega_{\widehat{I}^j(x,y)}(x,y) dy$ and
$\sigma_{n,i}(x)=\fint_{\{\delta(x,y)=2^i\}}\Omega_{I^{(n)}(x,y)}(x,y) dy$ and $\fint_E f$ denotes the mean value of $f$ on $E$. So that
\begin{equation*}
\int\limits_{y\in \mathbb{R}^+}Q_{l,0}(x,y)dy
= \sqrt{2}\sum_{i=0}^{-l-1}2^{-i}\widehat{\sigma}_{l,i+l}(x)
-\frac{\sqrt{2}}{2}\Biggl(\sum_{n=1}^{-l}2^{-n}\sum_{i=-n}^{-1}\sigma_{n,i}(x)+\sum_{n=-l+1}^{\infty}2^{-n}\sum_{i=l}^{-1}\sigma_{n,i}(x)\Biggr).
\end{equation*}
Since in the definitions of $\widehat{\sigma}$ and $\sigma$ we are taking mean values of functions with $L^{\infty}$-norm equal to $1$, we certainly
have that $\abs{\widehat{\sigma}}\leq 1$ and $\abs{\sigma}\leq 1$. Hence $\abs{\sum_{i=-n}^{-1}\sigma_{n,i}(x)}\leq n$, and
$\abs{\sum_{i=l}^{-1}\sigma_{n,i}(x)}\leq\abs{l}=-l$. So the first term in the expression for the integral is dominated by the geometric series
$\sum_{i\geq 0}2^{-i}$, the second term is dominated by the convergent series $\sum_{n=1}^{\infty}n2^{-n}$ and the third term is bounded by
$\abs{l}\sum_{n=-l+1}^{\infty}2^{-n}$ which tends to zero as $\abs{l}$ tends to infinity.
\end{proof}
Let us notice that (4) and (5) in the above lemma hold also integrating in the variable $x$.

One more remark is in order; $P$ is dyadicaly homogeneous of degree $-1$ and $\Omega$ of degree zero. In  other words
$P(2^jx,2^jy)=2^{-j}P(x,y)$ and $\Omega(2^jx,2^jy)=\Omega(x,y)$.

From the above lemma, we conclude that with
\begin{eqnarray*}
  \mathcal{P}^* f(x) &=& \sup_{\substack{l<m\\ l,m\in \mathbb{Z}}}\abs{\int_{\mathbb{R}^+}P^{l,m}(x,y)f(y) dy} \textrm{,\, and } \\
  \mathcal{P}_* f(x) &=& \sup_{\substack{l<m\\ l,m\in \mathbb{Z}}}\abs{\mathcal{P}_{l,m}(x,y)}
\end{eqnarray*}
we have
\begin{eqnarray}\label{eq:maximalOperatorsrelationMaximalHL}
\mathcal{P}_* f(x) &\leq& 4\sqrt{2} M_{dy} f(x)+\mathcal{P}^*f(x) \textrm{,\, and }\notag\\
\mathcal{P}^* f(x) &\leq& 4\sqrt{2} M_{dy} f(x)+\mathcal{P}_*f(x),
\end{eqnarray}
where
\begin{equation*}
M_{dy}f(x)=\sup_{x\in I\in\mathcal{D}^+}\frac{1}{\abs{I}}\int_I\abs{f(y)}dy
\end{equation*}
the dyadic maximal operator.

\section{Weighted norm inequalities for the Petermichl's operator}\label{sec:weights}

We shall see in this section that $\mathcal{P}$ satisfies all the conditions in \cite{Aimar} in order to show the $L^p(\mathbb{R}^+,wdx)$
boundedness for $w\in A_p(\mathbb{R}^+,\delta,dx)$ which coincides with the dyadic Muckenhoupt weights in $\mathbb{R}^+$. For
the sake of completeness we proceed to provide the statement of the main result in \cite{Aimar} on normal spaces of homogeneous type for
general Calder\'{o}n-Zygmund operators.

Let $X$ be a set. A quasi-distance on $X$ is a nonnegative and symmetric function $d$ on $X\times X$, vanishing only on the diagonal
of $X\times X$ such that for some $\kappa>0$ the inequality $d(x,z)\leq\kappa(d(x,y)+d(y,z))$ holds for every $x$, $y$ and $z\in X$. The
main results on the structure of quasi-metric spaces are contained in \cite{MaSe79Lip}. The Borel sets in $X$ are those in the $\sigma$-algebra
generated by the topology induced in $X$ by the neighborhoods defined by the $d$-balls. If the $d$-balls are Borel sets and $\mu$ is a positive
Borel measure such that for some constant $A$ the inequalities
\begin{equation*}
0<\mu(B(x,2r))\leq A\mu(B(x,r))<\infty
\end{equation*}
hold for every $x\in X$ and every $r>0$, where $B(x,r)=\set{y\in X:d(x,y)<r}$, we say the $(X,d,\mu)$ is a space of homogeneous type.

Let $(X,d,\mu)$ be a space of homogeneous type such that continuous functions are dense in $L^1(X,\mu)$. Let $1<p<\infty$, a nonnegative
and locally integrable function $w$ defined on $X$ is said to satisfy the Muckenhoupt $A_p$ condition, or $w\in A_p(X,d,\mu)$, if there exists a
constant $C$ such that
\begin{equation*}
\left(\fint_B wd\mu\right)\left(\fint_Bw^{-\tfrac{1}{p-1}}d\mu\right)^{p-1}\leq C
\end{equation*}
for every $d$-ball $B$. As before, $\fint_E wd\mu=\mu(E)^{-1}\int_Ew(x)d\mu(x)$. A weight $w$ is said to belong to $A_\infty$ if there exist two
constants $C$ and $\eta>0$ such that the inequality
\begin{equation*}
\frac{w(E)}{w(B)}\leq C\left(\frac{\mu(E)}{\mu(B)}\right)^{\eta}
\end{equation*}
holds for every ball $B$ and every measurable subset $E$ of $B$. The Hardy-Littlewood maximal function in this setting is, naturally, given by
\begin{equation*}
Mf(x)=\sup_{x\in B}\frac{1}{\mu(B)}\int_B\abs{f}d\mu.
\end{equation*}
The results in \cite{MaSe81} show the reverse H\"{o}lder inequality for $A_p$ weights and, as a consequence, the boundedness of the Hardy--Littlewood maximal
in $L^p(X,w d\mu)$ when $w\in A_p$.
\begin{theorem}[\cite{MaSe81}, \cite{AiMa84}]\label{thm:MaximalconditionApETH}
Let $(X,d,\mu)$ be a space of homogeneous type and $1<p<\infty$. Then $w\in A_p$ if and only if for some constant $C$ we have
\begin{equation*}
\int_X(Mf(x))^pw(x)dx\leq C\int_X\abs{f(x)}^pw(x)d\mu(x)
\end{equation*}
for every measurable function $f$.
\end{theorem}
For singular integrals, the detection of the correct integral singularity of the space is attained
after normalization of the space $(X,d,\mu)$ (\cite{MaSe79Lip}). We shall assume here that $(X,d,\mu)$ is a normal space in the sense that there
exist two constants $0<\alpha\leq \beta<\infty$ such that $\alpha r\leq\mu(B(x,r))\leq\beta r$. Let us only recall two particular instances of this situation. The
first, $X=\mathbb{R}^n$, $d(x,y)=\abs{x-y}^n$ and $\mu$ Lebesgue measure. The second, $X=\mathbb{R}^+$, $d(x,y)=\delta(x,y)=\abs{I(x,y)}$, where
$I(x,y)$ is the smallest dyadic interval containing $x$ and $y$. In this case $\mu$ is one dimensional Lebesgue measure.

The next statement collects the boundedness results for singular integrals in \cite{Aimar}.
\begin{theorem}[\cite{Aimar}]\label{thm:thesisAimar}
Let $(X,d,\mu)$ be a normal space such that continuous functions are dense in $L^1$. Assume that for every $r>0$ and every $x_0\in X$ we have that $\mu(B(x,r)\vartriangle B(x_0,r))\to 0$ when
$d(x,x_0)\to 0$, where $E\triangle F$ denotes the symmetric difference of $E$ and $F$. Let $T$ be a Calder\'{o}n-Zygmund operator on
$(X,d,\mu)$ in the sense of Definition~\ref{def:CalderonZygmundOperator} in \S\ref{sec:PetermichlOperatorAsCZ}. Let $K(x,y)$ be the kernel of $T$. Assume that the kernel $K$ satisfies also,
\begin{itemize}
\item[(iii)] for every $R>r>0$, we have
\begin{itemize}
\item[(iii.a)] $\abs{\int_{r\leq d(x,y)<R}K(x,y)d\mu(y)}$ is bounded uniformly in $r$, $R$ and $x$. 

Moreover, $\int_{r\leq d(x,y)<1}K(x,y)d\mu(y)$
converges uniformly in $x$ when $r$ tends to zero.
\item[(iii.b)] $\abs{\int_{r\leq d(x,y)<R}K(x,y)d\mu(x)}$ is bounded uniformly in $r$, $R$ and $y$. 

Moreover, $\int_{r\leq d(x,y)<1}K(x,y)d\mu(x)$
converges uniformly in $y$ when $r$ tends to zero.
\end{itemize}
\end{itemize}
Then, with $T_{R,r}f(x)=\int_{y\in X}K_{R,r}(x,y)f(y)d\mu(y)$, $K_{R,r}=\mathcal{X}_{r\leq d<R}K$ and $T_* f(x)=\sup_{\varepsilon>0}\abs{T_{\infty,\varepsilon}f(x)}$, we have
\begin{enumerate}
\item for $1<p<\infty$ there exists the $L^p(X,\mu)$ limit $Tf$ of $T_{R,r}f$ when $R\to +\infty$ and $r\to 0$;
\item for $f\in L^p(X,\mu)$ and $1<p<\infty$ we have Cotlar's inequality
\begin{equation*}
T_* f(x)\leq C M(Tf(x))+CMf(x);
\end{equation*}
\item the maximal operator $T_*$ is of weak type (1,1). In other words, for some constant $C>0$ we have
\begin{equation*}
\mu\left(\set{T_* f>\lambda}\right)\leq\frac{C}{\lambda}\norm{f}_{L^1};
\end{equation*}
\item for $w\in A_{\infty}(X,\mu)$
\begin{equation*}
\int_X [T_*f(x)]^pw(x)d\mu(x)\leq C\int_X[Mf(x)]^pw(x)d\mu(x);
\end{equation*}
\item for $w\in A_p(X,\mu)$ we have
\begin{equation*}
\int_X [T_*f(x)]^p w(x)d\mu(x)\leq C\int_X\abs{f(x)}^p w(x)d\mu(x).
\end{equation*}
\end{enumerate}
\end{theorem}
As a consequence of the above result and of the results in Section~\ref{sec:PetermichlOperatorAsCZ}, we
get the weighted boundedness of the maximal operators associated to Petermichl's kernel. We say that $w$ defined
on $\mathbb{R}^+$ is in $A^{dy}_p(\mathbb{R}^+,dx)$ if the inequality  $\left(\fint_I wd\mu\right)\left(\fint_I w^{-1/(p-1)}d\mu\right)^{p-1}\leq C$ holds for every $I\in \mathcal{D}^+$.

\begin{theorem}\label{thm:OperatorPApconditionnecessary}
For $1<p<\infty$ and $w\in A^{dy}_p(\mathbb{R}^+,dx)$ we have that $\mathcal{P}_*$ is bounded in $L^p(\mathbb{R}^+,wdx)$.
\end{theorem}
\begin{proof}
Let us check that we are in the hypothesis of Theorem~\ref{thm:thesisAimar}. As we already proved $X=\mathbb{R}^+$, $d=\delta$ and $\mu=$Lebesgue measure, provide a normal space in which $\delta$-Lipschitz functions are dense in $L^1(\mathbb{R}^+,dx)$. In order to prove that $\abs{B_{\delta}(x,r)\vartriangle B_{\delta}(x_0,r)}$ tends to zero when $x$ tends to $x_0$ for fixed positive $r$, just notice that when $\delta(x,x_0)<r/2$, $B_\delta(x,r)$ and $B_\delta(x_0,r)$ coincide. From Theorem~\ref{thm:PetermichlOperatorisCZ} we have the kernel $P(x,y)$ satisfies (i) and (ii) in the Definition of Calder\'{o}n--Zygmund operator. On the other hand, since $P^{l,m}=P_{l,m}+Q_{l,m}$ from (2), (4) and (5) in Lemma~\ref{lem:PlmQlm}, we get (iii) in Theorem~\ref{thm:thesisAimar}. Then we can apply Theorem~\ref{thm:thesisAimar} to obtain the boundedness properties of $\mathcal{P}_*$ in particular the weighted boundedness contained in (5). It only remains to observe that $A_p(\mathbb{R}^+,\delta,dx)=A^{dy}_p(\mathbb{\mathbb{R}}^+,dx)$
\end{proof}

\begin{theorem}
Let $1<p<\infty$. Then $\mathcal{P}^*$ is bounded in $L^p(\mathbb{R}^+,wdx)$ if and only if $w\in A^{dy}_p(\mathbb{R}^+,dx)$.
\end{theorem}
\begin{proof}
The sufficiency of $w\in A^{dy}_p(\mathbb{R}^+,dx)$ for the boundedness of $\mathcal{P}^*$ in $L^p(\mathbb{R}^+,wdx)$, $1<p<\infty$,
follows from \eqref{eq:maximalOperatorsrelationMaximalHL}, Theorem~\ref{thm:OperatorPApconditionnecessary} and Theorem~\ref{thm:MaximalconditionApETH}, since $Mf$ in $(\mathbb{R}^+,\delta,dx)$ is the dyadic Hardy-Littlewood maximal function $M_{dy}f$. Let us finally show that $A^{dy}_p(\mathbb{R}^+,dx)$ is necessary for the $L^p(\mathbb{R}^+,wdx)$. Assume that $w$ is a weight in $X$ such that $\mathcal{P}^*$ is bounded as an operator on $L^p(X,wd\mu)$. Since $\mathcal{P}^*f(x)\geq\abs{\sum_{I\in\mathcal{D},\abs{I}=\abs{I_0}}\proin{f}{h_I}(h_{I^-}(x)-h_{I^+}(x))}$ for
any $I_0\in\mathcal{D}$, taking $f=h_{I_0}w^{-1/(p-1)}$ we get
\begin{equation*}
\mathcal{P}^*f(x)\geq\Bigl<w^{-\tfrac{1}{p-1}}h_{I_0}, h_{I_0}\Bigr>\abs{h_{I^-_0}(x)-h_{I^+_0}(x)}
=\frac{1}{\abs{I_0}}\left(\int_{I_0}w^{-\tfrac{1}{p-1}}d\mu\right)\frac{\sqrt{2}}{\sqrt{\abs{I_0}}}\mathcal{X}_{I_0}(x).
\end{equation*}
Hence, from the inequality
$\int_X (\mathcal{P}^*f)^p wd\mu\leq C\int_X\abs{f}^p wd\mu$ that we are assuming, taking $f=h_{I_0}w^{-1/(p-1)}$ we get
\begin{equation*}
\frac{2^{p/2}}{\abs{I_0}^{3p/2}}\left(\int_{I_0}w^{-\tfrac{1}{p-1}}d\mu\right)^p w(I_0)
\leq C\frac{1}{\abs{I_0}^{p/2}}\int_{I_0}w^{-\tfrac{1}{p-1}}w d\mu
\end{equation*}
which implies that $w\in A^{dy}_p(\mathbb{R}^+,d\mu)$.
\end{proof}

As a final remark, let us observe that from the representation of the Hilbert kernel given in \cite{Petermichl00} and our result, we can get the well known weighted norm inequalities for the Hilbert transform.



\providecommand{\bysame}{\leavevmode\hbox to3em{\hrulefill}\thinspace}
\providecommand{\MR}{\relax\ifhmode\unskip\space\fi MR }
\providecommand{\MRhref}[2]{%
  \href{http://www.ams.org/mathscinet-getitem?mr=#1}{#2}
}
\providecommand{\href}[2]{#2}


\bigskip

\bigskip
\noindent{\footnotesize
\textsc{Instituto de Matem\'{a}tica Aplicada del Litoral, UNL, CONICET, FIQ.}

\smallskip
\noindent\textmd{CCT CONICET Santa Fe, Predio ``Alberto Cassano'', Colectora Ruta Nac.~168 km 0, Paraje El Pozo, S3007ABA Santa Fe, Argentina.}
}
\bigskip

\end{document}